\documentclass{amsart}
\usepackage{amsfonts}
\usepackage{amsmath}
\usepackage{amssymb}
\usepackage{amsthm}
\usepackage[latin1]{inputenc}

\numberwithin{equation}{section}

\DeclareMathOperator{\ord}{ord}
\DeclareMathOperator{\supp}{supp}

\DeclareMathOperator{\id}{id}

\DeclareMathOperator{\red}{red}

\newcommand{\Fc}{\mathcal{F}}
\newcommand{\Bc}{\mathcal{B}}
\newcommand{\Ac}{\mathcal{A}}
\newcommand{\Lc}{\mathcal{L}}

\newcommand{\cb}{\mathsf{c}}
\newcommand{\Lb}{\mathsf{L}}

\newcommand{\N}{\mathbb{N}}

\newtheorem{theorem}{Theorem}[section]
\newtheorem{proposition}[theorem]{Proposition}
\newtheorem{lemma}[theorem]{Lemma}
\newtheorem{corollary}[theorem]{Corollary}

\theoremstyle{definition}

\newtheorem{example}[theorem]{Example}

\begin{document}

\title[Minimal distances of plus-minus weighted zero-sum sequences]{The set of minimal distances of the monoid of plus-minus weighted zero-sum sequences and applications to the characterization problem}

\author{Kamil Merito \and  Oscar Ordaz  \and Wolfgang A. Schmid}

\address{(O.O.) Escuela de Matem\'aticas y Laboratorio MoST, Centro ISYS, Facultad de Ciencias, Universidad Central de Venezuela, Ap. 47567, Caracas 1041--A, Venezuela}
\address{(K.M., W.A.S.) Laboratoire Analyse, G{\'e}om{\'e}trie et Applications, LAGA, Universit{\'e} Sorbonne Paris Nord, CNRS, UMR 7539, F-93430, Villetaneuse, France
 \\ and \\ Laboratoire Analyse, G{\'e}om{\'e}trie et Applications (LAGA, UMR 7539) \\ COMUE  Universit{\'e} Paris Lumi{\`e}res \\  Universit{\'e} Paris 8, CNRS \\  93526 Saint-Denis cedex, France} 
\email{kamilmerito@hotmail.com}
\email{oscarordaz55@gmail.com}
\email{wolfgang.schmid@univ-paris8.fr, schmid@math.univ-paris13.fr}

\subjclass[2010]{11B30, 11B75, 11R04, 20K01}

\keywords{finite abelian group, weighted subsum, zero-sum problem}

\begin{abstract}
Recently a systematic investigation of monoids of sequences of plus-minus weighted zero-sum sequences had been started, which is among others motivated by applications to monoids of norms of algebraic integers. In the current paper these investigations are continued. The focus is on the set of minimal distances  of these monoids, which is an important arithmetical invariant. Applications to the characterization problem are discussed as well.    
\end{abstract}
\maketitle

\section{Introduction}

There is a wealth of literature on the arithmetic of various types of monoids. A central class are Krull monoids and related to this monoids of zero-sum subsequences over abelian groups. 

In a recent paper Boukheche and the authors of the current paper \cite{BMOS} considered monoids of weighted zero-sum sequences; while there are many investigations on zero-sum problems with weights, see for example \cite{adGS,adhi0,remarksPM}, the arithmetic of these monoids had not been investigated. Moreover, that paper established a connection to certain monoids of norms of algebraic integers. Indeed, these monoids of norms had first been studied under the term normsets by Coykendall \cite{Coykendall}; see \cite{CoykendallHasenauer} for a more recent exposition.  

In a subsequent paper Geroldinger, Halter-Koch, and Zhong \cite{GHKZ} generalized this connection and obtained further results on the arithmetic of monoids of weighted zero-sum sequences in particular on their (local) elasticities. A partciular emphasis was put on monoids of plus-minus weighted zero-sum sequences. For further results see \cite{Fabsitsetal,GeroldingerKainrath}.

In the current paper we further pursue the investigations of monoids of weighted zero-sum sequences. In particular, we study their sets of minimal distances, that is, the set formed by the minimal distances of their divisor-closed submonoids, denoted $\Delta^{\ast}(H)$; sometimes this set is also called the set of differences of the monoid, but this name is less common and risks to be confused with the set of distances $\Delta(H)$. This set has been studied a lot for Krull monoids, mostly for Krull monoids with finite class group (see \cite{GeroldingerZhongDelta,ZhongDelta}); for a result in the case of infinite class groups see \cite{ChapmanSchmidSmith}. 
 
Here, we investigate this set for monoids of plus-minus weighted zero-sum sequences over abelian groups, mainly for finite abelian groups. It turns out that the results depend strongly on the parity of the order of the group. In the case of groups of odd order (and not $1$), we obtain strong results on the sets  $\Delta^{\ast}(H)$; we show that  
$\max\Delta^{\ast}(H) = \exp(G)-2$  in that case and obtain strong constraints on its other elements that sometimes allow us to obtain a complete description. 
In the case of groups of even order, we show that $\max \Delta^{\ast}(H)$ can be quite large, it can exceed the value of the invariant for the monoid of zero-sum sequences without weights over the same group.

In the final section, we apply these results to characterize certain monoids of plus-minus weighted zero-sum problems by their systems of sets of lengths. 
In particular, we can show that the exponent of the group is determined by the system of sets of lengths for every group of odd order (excluding a few well-understood cases of very small groups). Then we apply this result to characterize elementary $p$-groups for $p$ odd and also obtain an alternative charaterization for cyclic groups of odd order (a first characterization was recently obtained in \cite{Fabsitsetal}).

\section{Preliminaries}

We recall some key-notions that we use frequently in the current paper. For the most part they are fairly  well-established in factorization theory  and we refer to \cite{GeroldingerHalterKochBOOK} for a complete exposition.

By $\mathbb{N}$ and $\mathbb{N}_0$ we denote the set of positive and non-negative integers, respectively. For real numbers $a$ and $b$, we denote by $[a, b] = \{ x \in  \mathbb{Z} \colon a  \le x \le b \}$ the interval of integers. 

\subsection{Abelian groups}
For $n \in \mathbb{N}$, let $C_n$ a cyclic group of order $n$; usually we use additive notation for abelian groups. 
It is well-known that for $(G,+,0)$ a nontrivial finite abelian group there are uniquely determined $1< n_1 \mid \dots \mid n_r$ such that $G \cong C_{n_1} \oplus \dots \oplus C_{n_r}$ where $r$ is the rank of $G$, denoted by $\mathsf{r}(G)$, and $n_r$ is the exponent of $G$, denoted by $\exp(G)$.  The exponent of a group of order $1$ is $1$ and its rank is $0$. If the exponent is a power of a prime $p$, then we call $G$ a $p$-group, and we say that it is an elementary $p$-group if the exponent is $p$.

For $m \in \mathbb{N}\setminus \{1\}$, let $\mathsf{r}_m(G) = |\{ i \colon m \mid n_i\}|$ denote the $m$-rank of $G$; for the most part we use this for $m$ a prime number.  

A family $(e_1, \dots, e_s)$ over $G$ is called independent if $\sum_{i=1}^s a_i e_i= 0$ with $a_i \in \mathbb{Z}$ implies that $a_ie_i =  0$ for each $i \in [1,s]$. 
We also use the terminology independent for subsets of $G$ in the obvious way.  
 
\subsection{Monoids} Throughout the paper,  monoid always means a semigroup with identity that is commutative and satisfies the cancellation law. We usually use multiplicative notation for monoids and denoted the identity element by $1_H$ or simply $1$ if there is no risk of confusion. 

An element $a$ of a monoid $H$ is called invertible or a unit if there exists an element $a' \in H$ such $aa'= 1_H$; the set of invertible elements of $H$ is denoted by $H^{\times}$; it is a subgroup of $H$.  A reduced  monoid is a monoid whose only invertible element is  $1_H$;  one calls $H_{\red}=H/H^{\times}$ the reduced monoid associated to $H$. 

We say that $b$ divides $a$ in $H$, and we write $b\mid_H a$ if there exists some $c \in H$ such that $a= bc$. If there is no risk of confusion we drop the subscript $H$. If $bc= a$, then we write $b^{-1}a$ for $c$. 

An element $a\in H$ is called irreducible or an atom (in $H$) if it has no proper divisors, that is,  $a = bc$ with $b,c \in H$ implies that $b$ or $c$ are invertible. The set of irreducible elements of $H$ is denoted by $\mathcal{A}(H)$. An element $p \in H \setminus H^{\times}$ is called prime if $p \mid bc$ with $b,c \in H$ implies that $p \mid b$ or $p \mid c$. 

A submonoid $S$ of $H$ is called a saturated submonoid if for $s,t \in S$ one has $s \mid_S t$ if and only if  $s\mid_H t$    (of course one implication is trivial) and it is called a divisor closed submonoid  if for $s \in S$ and $a \in H$, one has that $a\mid_H s $ implies that $a \in S$. 

For a set $P$, let $\Fc(P)$ denote the free commutative monoid over $P$, as is common in the context of factorization theory we refer to its elements as sequences. 
To explain the terminology `sequence' one can note that it is a collection of elements of $P$ where repetitions are possible yet the ordering of  elements is disregarded, in that sense they are finite sequences where the ordering is disregarded.  

Essentially by definition, for $S\in \Fc(P)$, there are unique $v_p \in \mathbb{N}_0$, all but finitely many equal to $0$, such that $S= \prod_{p \in P}p^{v_p}$; the  integer $v_p$ is called the multiplicity of $p$ in $S$ or also the $p$-adic valuation of $S$, denoted by $\mathsf{v}_p(S)$.
Moreover, there exist up to ordering uniquely determined $p_1, \dots, p_{\ell} \in P$ (not necessarily distinct) such that $S= p_1 \dots p_{\ell}$. For a sequence $S$ a divisor $T$ of $S$ in $\mathcal{F}(P)$ is called a subsequence of $S$.  
The identity element of the monoid of sequences is called the  empty sequence, and is denote it by $1$ unless there is a risk of confusion. 
One calls $\ell$ the length of $S$, it is denoted by $|S|$. Moreover, the set $\{p_1, \dots , p_{\ell}\}$ is called the support of $S$, denoted $\supp(S)$; its cardinality is bounded above by $\ell$ yet it is usually strictly smaller as repeated elements are counted only once. 

In the current paper we are mainly interested in the case of sequences over abelian groups or subsets of abelian groups. Let  $G_0$ be a subset of an abelian group $(G, +)$. In this case, for a sequence $S= g_1 \dots g_{\ell} \in \Fc(G_0)$, we denote by $\sigma(S) = \sum_{i=1}^{\ell} g_i \in G $ its sum and by  $\sigma_{\pm}(S) = \{ \sum_{i=1}^{\ell} \epsilon_ig_i  \in G  \colon \epsilon_i \in \{+1, -1\}\}$ the set of plus-minus weighted sums of $S$. 
It is possible to consider other sets of weights, yet in the current paper we deal exclusively with this sets of weights. As an aside, for the case of general weights one takes endomorphisms of $G$ rather than integers as weights; the current set of weights would thus correspond to $\{\id_G, -\id_G\}$.  
 
If  $\sigma(S) = 0$, the sequence is called a zero-sum sequence and if $0 \in \sigma_{\pm}(S)$,  the sequence is called a plus-minus weighted  zero-sum sequence.  The set of all sequences over $G_0$ that are zero-sum sequences is denoted by $\Bc(G_0)$, and the set set of all plus-minus weighted  zero-sum sequence  is
denoted by $\Bc_{\pm}(G_0)$. It is easy to see that both are submonoids of $\Fc(G_0)$,  and that $\Bc(G_0)$ is a submonoid of $\Bc_{\pm}(G_0)$. 

It is important to note though that while $\Bc(G_0)$ is a saturated submonoid of $\Fc(G_0)$, it is usually not true that $\Bc_{\pm}(G_0)$ is a saturated submonoid of $\Fc(G_0)$. This is quite relevant because a monoid $H$ that is a saturated submonoid of a factorial monoid is a Krull monoid. Thus, while  $\Bc(G_0)$ is a Krull monoid this usually is not the case for $\Bc_{\pm}(G_0)$. Indeed, in case $G_0 = G$, it is known that  $ \mathcal{B}_{\pm}(G)$ is Krull if and only if $G$ is an elementary $2$-group and it thus actually is equal to  $ \mathcal{B}(G)$; for all other groups the monoids are not even transfer Krull   (see below for a definition and \cite{BMOS} for the result).  
 
While Krull monoids are arguably the most intensely studied class of monoids in factorization theory, here we do not go into much detail and refer to, for example, \cite{GeroldingerMonthly, GZ_survey, schmid_surveySoL} for survey articles.

\subsection{Factorizations} We continue by recalling some notions from factorization theory. A monoid  $H$ is called atomic if each non-invertible element of $H$ can be written as a (finite) product of irreducible elements. The monoid of factorizations of $H$, denoted by $\mathsf{Z}(H)$, is the monoid $\Fc(\Ac(H_{\red}))$. 

The homomorphism $\pi_H: \mathsf{Z}(H) \to H_{\red}$, which maps the formal product $a_1 \dots a_k$ to its value, 
is called the factorization homomorphism. It is surjective if and only if $H$ is atomic.  For $a \in H$, the set $\mathsf{Z}_H(a) = \pi_H^{-1}(a H^{\times})$ is called the set of factorizations of $a$ in $H$; if the monoid $H$ is obvious from context we drop $H$ from  the notation.  
 
We stress that  the  monoid of factorizations formalizes the idea that one identifies factorizations of elements of $H$  that differ only by the ordering of the terms or multiplication by units.  

For $z \in \mathsf{Z}(H)$, one calls $|z|$, which is defined as $\mathsf{Z}(H)$ is a free monoid, the length of the factorization; informally it is the number of irreducible elements in the factorization where multiplicities are taken into account. Moreover, one calls $\mathsf{L}_H(a) = \{ |z|  \colon z \in \mathsf{Z}_H(a)\}$ 
the set of lengths of $a$ in $H$. Finally, the system of set of lengths of $H$ is defined as $\mathcal{L}(H) = \{ \mathsf{L}_H(a) \colon a \in H\}$.  

The monoid is called a bounded factorization monoid, BF-monoid for short, if $\mathsf{L}_H(a)$ is finite for each $a \in H$. 
Similarly, monoids for which even the sets of factorizations are finite are called finite factorization monoids. 
We note that $H$ is factorial if and only if $|\mathsf{Z}_H(a)| = 1$ for each $a \in H$. One says that $H$ is half-factorial if $|\mathsf{L}_H(a)| = 1$ for each $a \in H$.

We end the section by briefly recalling the concept of transfer homorphism, see for example, \cite[Section 1.3]{GeroldingerBarcelona}. 
For a generalization to not necessarily commutative monoids see \cite{BaSe}. 

Let $H$ and $\mathcal{B}$ be monoids. A monoid homomorphism $\Theta: H \to \mathcal{B}$ is a transfer homomorphism if it satisfies the following:  
\begin{itemize}
\item[T1] $\mathcal{B} = \Theta(H)\mathcal{B}^{\times}$ and $\Theta^{-1}(\mathcal{B}^{\times}) = H^{\times}$. 
\item[T2] If $u \in H$ and $b,c \in \mathcal{B}$ with $\Theta(u) = bc$, then there exist $v,w \in H$ such that $u =vw$, $\Theta(v)  \simeq b$ and  $\Theta(w) \simeq c$.
\end{itemize}

A monoid (not necessarily commutative) is called transfer Krull if it admits a tranfer homomorphism to a Krull monoid.

The relevance of the notion of transfer homomorphism is due to the fact that it preserves many arithmetical properties. In particular, if $\Theta: H \to \mathcal{B}$ is a transfer homomorphism, then $\Lb_H(a) = \Lb_{\mathcal{B}}(\Theta(a))$ for each $a \in H$, and in particular $\Lc(H)= \Lc(\mathcal{B})$. 
There are various arithmetical invariants that are derived from sets of lengths and that are thus also preserved. We recall some of them in the next section.

\section{Structure of sets of lengths}
\label{sec_SSL}

In the current section we recall a few notions and results on sets of lengths, focusing on the context at hand.  For a more complete presentation see, e.g., \cite{GeroldingerMonthly,GeroldingerHalterKochBOOK,GZ_survey}. 

Let $H$ be a BF-monoid. To avoid trivial corner cases we assume that $H \neq H^{\times}$.
  
For $k \in \mathbb{N}$ one sets $\rho_k(H) = \sup \{\max \mathsf{L}_H(a) \colon k \in \mathsf{L}_H(a), \, a\in H \} \in \mathbb{N}_0 \cup \{\infty\}$, that is, $\rho_k(H)$ is (in case it is finite) the largest number of factors that can appear in the factorization into irreducibles of an element that is the product of $k$ irreducible elements.  
Obviously, $\rho_1(H) = 1$ and $\rho_k(H) \ge k$ for each $k$. The value $\rho_k(H)$ is sometimes called the $k$-th local elasticity of $H$. This terminology is derived from that of the elasticity of a monoid, denoted $\rho(H)$. If $A \subseteq \N$ is a finite subset, we call $\rho(A) = \max A / \min A \in \mathbb{Q}_{ \ge 1}$ the elasticity of $A$, 
and we set $\rho (\{0\}) = 1$. The elasticity of an element $a \in H$, denoted $\rho(a)$, is just $\rho(\mathsf{L}_H(a))$. 
Finally, the elasticity of $H$, denoted $\rho(H)$, is defined as $\sup \{ \rho(a) \colon a \in H \} \in \mathbb{R}_{ \ge 1} \cup \{\infty\}$. 
It is not difficult to see that $\rho(H)  = \lim_{k \to \infty} \rho_k(H) / k$.

Another important notion to investigate sets of lengths are sets of distances.  
For $A \subseteq \mathbb{Z}$ a finite subset, we denote by $\Delta (A)$ the set of (successive) distances of $A$, that is, for $A  = \{a_0, \dots, a_k\} $ with $a_0 < \dots < a_k$ one has $\Delta(A) =  \{ a_1  - a_0, a_2 - a_1,  \dots, a_k - a_{k-1}\} =   \{ a_{i + 1 } - a_i \colon i \in [1,k] \} $. 
Thus, for each $a \in A$, and $a$ not the maximum of $A$, we have $d \in \Delta(A)$ if $d$ is the smallest $d \in \N$ such that 
$a +d \in A$. 

For  $a \in H$, we set $\Delta(a) = \Delta (\Lb(a))$ and  
\[\Delta(H) = \bigcup _{a \in H} \Delta (\Lb(a)).\]
We observe that $H$ is half-factorial if and only if $\rho(H) = 1$ if and only if $\Delta (H) = \emptyset$. 
It is well-known that if $\Delta(H) \neq \emptyset$, then  $\min \Delta(H) = \gcd \Delta(H)$. 

Moreover, we set 
\[
\Delta^{\ast}(H)  = \{\min \Delta (S)  \colon S \subseteq H \text{ is a divisor closed submonoid with $\Delta(S) \neq \emptyset$}\}
\]
the set of minimal distances of $H$. The relevance of the notion $\Delta^{\ast}(H)$ is due to the fact that it is a natural choice in the structure theorem for sets of lengths, which we recall below.

Specifically, for $H$ a finitely generated monoid and not half-factorial,  there is some $M \in \mathbb{N}_0$ such that each set of lengths $L$ of $H$ is an almost arithmetical multiprogression with bound $M$ and difference $d \in \Delta^{\ast}(H) $, that is, 
\[
L  =  y + (L_1 \cup L^{\ast} \cup (\max L^{\ast} + L_2)) \subseteq y + \mathcal{D} + d\mathbb{Z}
\]
with  $y \in \mathbb{N}_0$,  $\{0,d\} \subseteq \mathcal{D}  \subseteq [0,d]$,  $-L_1, L_2 \subseteq [1,M]$, $\min L^{\ast} = 0$ and $L^{\ast} = [0,\max L^{\ast}] \cap \mathcal{D} + d\mathbb{Z} $.

The set $\Delta^{\ast}(H)$ is not the only choice for which this  result holds, indeed it obviously would hold for any superset of that set and in the other direction, if $d,d' \in\Delta^{\ast}(H)$ with $d' \mid d$, then possibly altering the constant $M$ slightly, the result would also hold for $\Delta^{\ast}(H) \setminus \{d'\}$. 

However, the set $\Delta^{\ast}(H)$ is a natural choice in that it is a set that naturally comes out of a proof of this result. Very roughly speaking, 
for an element that has `many' irreducible factors there ought to be some irreducibles that occur with a `high' multiplicity and they will induce a repeating pattern corresponding to the minimal distance of the divisor-closed submonoid they generate. 

Or, looked at it from another perspective, considering an element that is a product of multiple copies of all irreducible elements from a divisor-closed submonoid, one can see that its set of lengths will be to an almost arithmetic progression whose difference is the minimal distance of the submonoid. 

Indeed, if one denotes by $\Delta_1(H)$ the set of all $d \in \mathbb{N}$ such that $\mathcal{L}(H)$ contains almost arithmetical progression with difference $d$ of arbitrarily large length, then it is known that $\Delta^{\ast}(H) \subseteq \Delta_1(H)$ and each $d_1 \in \Delta_1(H)$ divides some $d \in \Delta^{\ast}(H)$. 
Recall that an  almost arithmetic progression  is an almost arithmetic multiprogression with $\mathcal{D} = \{0,d\}$, that is an arithmetic progression where some elements at the start and at the end might be missing. 

While we mentioned the result for finitely generated monoids, this is not the only class of monoids for which the result holds; one way to extend the result, yet not the only way, is to use tranfer homomorphisms to finitely generated monoids.

We recall a few well-known facts that are useful in the remainder of the paper.  Let $H$ be a BF-monoid with $\Delta(H)\neq \emptyset$. We already recalled that $\min \Delta(H) = \gcd \Delta(H)$, and thus  in particular, for every choice of $a \in H$ and $l, l' \in \Lb(a)$, one has $\min \Delta(H) \mid l' - l$.

Moreover, if $S$ is a divisor-closed submonoid of $H$, then $ \Delta(S) \subseteq \Delta (H)$ and thus $\min \Delta(H) \le \min \Delta(S) $ and indeed 
$\min \Delta(H) \mid \min \Delta(S)$.

The results recalled hitherto apply to all finitely generated monoids and it is not hard to see that monoids of plus-minus weighted zero-sum sequences over subsets of finite abelian groups are finitely generated, and thus the results apply.
More precisely, for $G$ a finite abelian group and $H = \Bc_{\pm}(G)$ one has that  $\{|A| \colon A \in \Ac(H)\}$ is bounded above by the Davenport constant of $G$, denoted $ \mathsf{D}(G)$. Moreover, $\max\{|A| \colon A \in \Ac(H)\}$ is called the Davenport constant of the monoid $H$, denoted $\mathsf{D}(H)$; this definition is used for every monoid embedded into a free monoid, not just monoids of (weighted) zero-sum sequences.  

We recall that the Davenport constant of a finite abelian group $G$, denoted $\mathsf{D}(G)$, is defined as the smallest integer $\ell$ such that each sequence of lenght at least $\ell$ admits a subsequence that is a zero-sum subsequence. It is well-known and easy to see that it is equal to  $\mathsf{D}(\Bc(G))$ the Davenport constant of the monoid $\Bc(G)$. 
It is well known that for $G \cong C_{n_1} \oplus \dots \oplus C_{n_r}$ with  $1< n_1 \mid \dots \mid n_r$ one has $|G| \ge \mathsf{D}(G)\ge \sum_{i=1}^r (n_i-1) + 1$, and the latter quantity is denoted by $\mathsf{D}^{\ast}(G) $. Indeed, for $p$-groups and groups of rank at most two, one knows that  $\mathsf{D}(G)= \sum_{i=1}^r (n_i-1) + 1$, yet this is not true in general.  

We point out that, as is the case without weights, the situation is very different for infinite abelian groups. Indeed, Geroldinger and Kainrath \cite{GeroldingerKainrath} recently showed that for $G$ an infinite abelian group every finite subset of $\mathbb{N}_{\ge 2}$ is a set of lengths of some element of $H=\Bc_{\pm}(G)$, just like for $\Bc(G)$. Note that since these monoids are still BF-monoids and the only sets of lengths that contain elements less than $2$ are $\{0\}$ and $\{1\}$, this yields indeed a complete description of the system of sets of lengths of $\Bc_{\pm}(G)$ for infinite $G$. We recall that by a result of Kainrath \cite{Kainrath} that was already  known for  $\Bc(G)$. 

However, while $\Delta^{\ast} (\Bc(G))  = \mathbb{N}$ (see \cite{ChapmanSchmidSmith}) this is not always the case for $\Delta^{\ast} (\Bc_{\pm}(G))$; this was already noted in \cite{GeroldingerKainrath} and we provide further results on that problem. 

We close this section with observations and results specific to monoids of plus-minus weighted zero-sum sequences.  
Since we want to study $\Delta^{\ast}(H)$, which is defined using divisor-closed submonoids of $H$, it is crucial to know the divisor-closed submonoids.  
Let $G$ be an abelian group. While it is immediate that for $G_0 \subseteq G$, the monoid $\Bc_{\pm}(G_0)$ is a divisor-closed submonoid, there might in principle be other divisor-closed submonoids. However, it turns out that this is not the case. We include a short proof for the particular case; the result if obtained for arbitrary sets of weights for finite abelian groups in \cite{Fabsitsetal}.  

\begin{lemma}
\label{lem_divclosed}
Let $G$ be an abelian group. Then $\Bc_{\pm}(G_0)$ with $G_0 \subseteq G$ a subset are all the divisor-closed submonoids of $\Bc_{\pm}(G)$.
\end{lemma}
\begin{proof}
It is clear that $\Bc_{\pm}(G_0)$ for $G_0 \subseteq G$ is a divisor-closed submonoid of $\Bc_{\pm}(G)$. 
Assume that $H \subseteq \Bc_{\pm}(G)$ is divisor-closed. We need to show that there exists some  $G_0 \subseteq G$ such that $H = \Bc_{\pm}(G_0)$.

Set $G_0 = \bigcup_{B\in H} \supp(B)$. Then  $H \subseteq  \Bc_{\pm}(G_0)$ and we need to show the converse inclusion. 
Let $S \in  \Bc_{\pm}(G_0) $, say, $S= g_1 \dots g_k$  with $g_i \in G_0$ and let $\epsilon_i \in \{+1,-1\}$ such that $\sum_{i = 1}^k \epsilon_i g_i = 0$. 
By the definition of $G_0$, for each $g_i$ there exists some $T_i \in H$ such that $g_i$ occurs in $T_i$. 

Now $S$ is a subsequence of $T_1 \dots T_k$. It remains to show that $S$ divides $T_1 \dots T_k$ in $\Bc_{\pm}(G)$, that is, $S^{-1}T_1 \dots T_k \in \Bc_{\pm}(G)$. 

Let $T_i'= g_i^{-1} T_i$. Then $S^{-1}T_1 \dots T_k = T_1' \dots T_k'$. Since each $T_i$ is a plus-minus weighted zero-sum sequence we know that $g_i$ or $-g_i$ is in $-\sigma_{\pm}(T_i')$, yet flipping all the weights we see that indeed both are contained in this set. Therefore for each $i$ we get that $\epsilon_i g_i \in \sigma_{\pm}(T_i')$, which implies that  $\sum_{i = 1}^k \epsilon_i g_i  \in \sigma_{\pm}(T_1' \dots T_k')$. Yet $\sum_{i = 1}^k \epsilon_i g_i$ is $0$ by assumption and  $T_1' \dots T_k'$ is indeed a plus-minus weighted zero-sum sequence. 
\end{proof}

\section{Minimal distances}

In the current section we carry out our investigation of sets of distances, in particular their minima. 

We first recall a result essentially in  \cite{GHKZ}.  

\begin{lemma} 
\label{min=1}
Let $G$ be a finite abelian group and let $H = \Bc_{\pm} (G)$. 
Then $\Delta   (H) = \emptyset$ if and only if $|G| \le 2$.
If $|G| \ge 3$, then $1\in \Delta(H)$ and in particular $\min \Delta^\ast   (H) = \min \Delta   (H) = 1$. 
\end{lemma} 
\begin{proof}
The first assertion is Lemma 6.1 in \cite{GHKZ}, that $\min \Delta   (H) = 1$ is implicit in its proof, and the assertion on $\min \Delta^\ast(H)$ follows as $H$ is a divisor-closed submonoid of itself. 
\end{proof}

We now give  a simple yet powerful lemma; the restriction that $0 \notin G_0$ is no actual restriction; 
since $0$ is a prime element of $\Bc_{\pm} (G)$, it is easy to see that  $\Bc_{\pm} (G_0)$ and $\Bc_{\pm} (G_0 \cup \{0\})$ have the same set of distances. 

\begin{lemma}
\label{lem_Delta_div}
Let $G$ be an abelian group and let $G_0 \subseteq G \setminus \{0\}$ be non-empty. Let $H = \Bc_{\pm} (G_0)$.
For each $A \in \mathcal{A}(H)$, we have that $\{2, |A|\} \subseteq \mathsf{L}(A^2)$. In particular, if $\Delta (H)\neq \emptyset $, then $\min \Delta (H) \mid |A|-2$. 
Moreover,  
\[\min \Delta (H) \mid  \gcd \{ |A|-2 \colon A \in \mathcal{A}(H)\} = \gcd \{ |A|- |A'| \colon A, A' \in \mathcal{A}(H)\}. \]
\end{lemma}
\begin{proof}
Let $A = g_1 \dots g_l \in  \mathcal{A}(H)$. Note that since $g_i \neq 0$, we have $l \ge 2$ and $g_i^2 \in \mathcal{A}(H)$ for each $i \in [1, l]$. 
Thus $A^2$ and  $g_1^2 \dots g_l^2$ are factorizations of the same element, the former has length $2$ and the latter length $l = |A|$.  
 While $|A|-2$ might not be in $\Delta(H)$ it is at least a sum of elements from $\Delta(H)$, and since $\min \Delta (H) = \gcd \Delta (H)$ it follows that  
$\min \Delta (H) \mid |A|-2$. 

Furthermore, as  $\min \Delta (H) \mid |A|-2$ for each $A \in \mathcal{A}(H)$, it divides also $\gcd \{ |A|-2 \colon A \in \mathcal{A}(H)\}$. 
Finally $\gcd \{ |A|-2 \colon A \in \mathcal{A}(H)\} = \gcd \{ |A|- |A'| \colon A, A' \in \mathcal{A}(H)\}$ holds as there is some $A'$ of length $2$ and by elementary properties of the greatest common divisor. 
\end{proof}

This argument is clearly specific to the case of plus-minus weighted zero-sum sequences. It is however possible to view it as a particular instance of a more general argument that would also encompass a related argument involving cross numbers for the the classical case, which we do not develop here.  

A natural question is to what extent $\min \Delta (H)$ and  $ \gcd \{ |A|-2 \colon A \in \mathcal{A}(H)\}$ can differ.
It turns out  that they are tightly linked. Before turning to this question we give a few additional general results. 
In particular, we characterize when $\Delta (H)$ is empty.

\begin{lemma} 
\label{lem_pm_hfdelta}
Let $G$ be an abelian group and let $G_0 \subseteq G$. Let $H = \Bc_{\pm} (G_0)$.
\begin{enumerate}
\item The monoid $H$ is half-factorial if and only if $\mathsf{D}(H) \le 2$. 
\item If $\Delta (H) \neq \emptyset$ and $\mathsf{D}(H)$ is finite,  we have $\min \Delta (H) \mid \mathsf{D}(H)-2$. 
\end{enumerate}
\end{lemma} 
\begin{proof}
By Lemma \ref{lem_Delta_div} it is immediate that $\mathsf{D}(H) > 2$ implies that $\Delta(H)$ is non-empty. 
To see the converse it suffices to note that the only element of $\Ac(H)$ of length $1$, is the prime element $0$. 

The second part also follows by Lemma \ref{lem_Delta_div} as there is an $A \in\Ac(H) $  with $|A| = \mathsf{D}(H)$. 
\end{proof}

This allows us to establish the following bound. 
 
\begin{corollary} 
\label{cor_delta_ast}
Let $G$ be an  abelian group and let $G_0 \subseteq G$. Let $H = \Bc_{\pm} (G_0)$.
We have $\max \Delta^{\ast} (H) \le \mathsf{D}(H)-2$. 
\end{corollary} 
\begin{proof}
If $\mathsf{D}(H)$ is infinite, the claim is trivial. Thus assume that $\mathsf{D}(H)$ is finite. 
By Lemma \ref{lem_pm_hfdelta} we know that $\min \Delta (H) \mid \mathsf{D}(H)-2$ and the same holds true for each divisor-closed submonoid $H'$ of $H$. 
It suffices to observe that  $\mathsf{D}(H') \le \mathsf{D}(H)$.  
\end{proof}

\begin{lemma}
\label{minDeltaH}
Let $G$ be an abelian group and let $G_0 \subseteq G \setminus \{0\}$ be non-empty. Let $H = \Bc_{\pm} (G_0)$. 
Then 
\[ \min \Delta (H) \mid \gcd \{ |A|-2 \colon A \in \mathcal{A}(H)\}  \mid 2 \min \Delta (H). \]
In particular, if $\gcd \{ |A|-2 \colon A \in \mathcal{A}(H)\}$ is odd, then  $\min \Delta (H) = \gcd \{ |A|-2 \colon A \in \mathcal{A}(H)\}$.   
\end{lemma}
\begin{proof}
The first divisibility relation is just Lemma \ref{lem_Delta_div}. We need to show that $d=  \gcd \{ |A|-2 \colon A \in \mathcal{A}(H)\}  \mid 2 \min \Delta (H)$.
Let $B \in H$ such that $B$ has a factorization of length $k$ and of length $l = k + \min \Delta (H)$, 
say $B = A_1 \dots A_k = C_1 \dots C_l$  with $A_i, C_j \in \mathcal{A}(H)$.

As $\sum_{i = 1}^k  |A_i| = \sum_{j = 1}^l  |C_j|$, we have $\sum_{i = 1}^k  |A_i| \equiv \sum_{j = 1}^l  |C_j| \pmod{d}$. 
Now, essentially by the definition of $d$, the length of each $A \in \mathcal{A}(H)$  is congruent to $2$ modulo $d$. 
Therefore $2 k \equiv 2 l \pmod{d}$, which implies $2 \min \Delta (H) \equiv 0 \pmod{d}$ as claimed. 
Of course in case $d$ is odd, this further yields $\Delta (H) \equiv 0 \pmod{d}$.
\end{proof}  

It might be interesting to observe in case $\gcd \{ |A|-2 \colon A \in \mathcal{A}(H)\}$ is even, there actually can be a discrepancy. 
A simple example can be obtained by considering elementary $2$-groups. Recall that in this case, plus-minus weighted zero-sum sequences are just classical zero-sum sequences. 

\begin{example}
Let $ G = C_2^4 = \langle e_1, e_2, e_3, e_4 \rangle$. Let $G_0  = e_1 + \langle e_2, e_3, e_4 \rangle$ and $H = \Bc_{\pm} (G_0)$.
Then $\min \Delta (H) = 1$ while  $\gcd \{ |A|-2 \colon A \in \mathcal{A}(H)\}=2$. 
\end{example}
To see that this is indeed the case, it suffices to note on the one hand that $|A|$ is necessarily even for $A \in \mathcal{A}(H)$, which is clear considering the projection to $\langle e_1 \rangle$. On the other hand, considering $A_0= (e_1+e_2 + e_3+ e_4)(e_1+e_2)(e_1 + e_3)(e_1+e_4)$, $A_1 = (e_1 + e_2 + e_3 + e_4)(e_1 + e_2)(e_1 + e_3)e_1$ and $A_2= (e_1 + e_2 + e_3 + e_4)(e_1 + e_2+e_3)(e_1 + e_4)e_1$, one has $A_1A_2 = A_0  \cdot e_1^2  \cdot (e_1 + e_2 + e_3 + e_4)^2$, which yields a distance of $1$. 

Furthermore, this is not an isolated phenomenon and not limited to the minimal distance $1$. For instance \cite[Theorem 5.6]{WHF} yields examples of other minimal distances for elementary $2$-groups.

We recall that if $G_1 \subseteq G_0$,  then  $\min \Delta ( \Bc_{\pm}(G_0)) \mid \min \Delta ( \Bc_{\pm}(G_1))$. 

In view of this studying minimal distances for small sets is of particular relevance. The following lemma gives a complete answer for singletons. 
The result depends on the parity of the order of the element.

\begin{lemma} 
\label{lem_onegenerated_ar}
Let $G$ be an abelian group. Let $g \in G \setminus \{0\}$ and let $H = \Bc_{\pm} (\{g\})$.
\begin{enumerate} 
\item If the order of $g$ is even or infinite, then $H$ is isomorphic to $(\N_0,  +)$. In particular, $\Delta (H)= \emptyset $.  
\item If the order of $g$ is odd, then $H$ is isomorphic to the numerical monoid $ \langle 2, \ord (g) \rangle \subseteq (\N_0,  +)$. In particular, $\Delta (H)= \{\ord(g) - 2\} $.  
\end{enumerate} 
\end{lemma} 
\begin{proof} 
We determine the irreducible elements in the respective monoids. 
In case the order of $g$ is infinite, it is easy to see that $g^2$ is the only plus-minus weighted minimal zero-sum sequence over $\{g\}$. 
This is also the case when the order of $g$ is even; just note that $g^{\ord(g)}$ is not a plus-minus weighted minimal zero-sum sequence if $\ord(g)$ is even and not $2$. 

Thus in these two cases, the monoid has a unique irreducible element and thus is isomorphic to $(\N_0,  +)$.  
By contrast, if $\ord(g)$ is odd, then there are two  plus-minus weighted minimal zero-sum sequences, namely $g^2$ and $g^{\ord(g)}$; while the latter does have proper subsequences that are also plus-minus weighted zero-sum sequences it cannot be decomposed into two nonempty plus-minus weighted zero-sum sequences.

The isomorphism to $ \langle 2, \ord (g) \rangle$ is obtained by considering the multiplicity of $g$, equivalently the length of the sequence. 
The result on the set of distances follows noting that the only minimal relation among the irreducibles is $(g^2)^{\ord(g)} = (g^{\ord(g)})^2$ or by invoking results on distances in numerical monoids (see \cite{minDeltaNum}).
\end{proof} 

\begin{lemma} 
\label{lem_delta}
Let $G$ be an abelian group and let $G_0 \subseteq G$. Let $H = \Bc_{\pm} (G_0)$.
We have $\{ \ord (g)  - 2\colon g \in G_0, \, \ord (g) \ge 3 \text{ odd}\} \subseteq \Delta (H) $. 
\end{lemma} 
\begin{proof} 
Let $g \in G_0$ be a nonzero element of odd order. Let $S = \Bc_{\pm} (\{g\})$. Then $S \subseteq H$ is a divisorclosed submonoid and thus $\Delta (S) \subseteq \Delta (H)$.  By Lemma \ref{lem_onegenerated_ar} we  know that $\Delta (S)= \{\ord (g) - 2\}$ and the claim follows.
\end{proof}

\subsection{Groups of odd order}

In the following result we obtain quite precise information on $\Delta^{\ast} (H)$ for  $H = \Bc_{\pm} (G)$ where $G$ is a group of odd order.

\begin{theorem} 
\label{DeltaAstOdd}
Let $G$ be a finite abelian group exponent $n$ and let $H = \Bc_{\pm} (G)$. 
Assume that $n$ is odd and at least $3$. 
Let $D_1 = \{ d - 2 \colon d \mid n, \, d \ge 3 \}$ and let $D_2 = \{ d'  \colon d' \mid d, \, d \in D_1 \}$.
Then $D_1 \subseteq \Delta^{\ast}   (H) \subseteq D_2$. 
In particular, $\max \Delta^{\ast}   (H) = n - 2$. 
\end{theorem} 
\begin{proof} 
To show that $D_1 \subseteq \Delta^\ast (H)$ it suffices to note that for each $d \mid n$ there exists an element $g \in  G$ of order $d$ and to invoke Lemma \ref{lem_delta}. To see that $\Delta^\ast (H) \subseteq D_2$, let $S \subseteq H$ be a divisor-closed submonoid, that is, $S = \Bc_{\pm} (G_0) $ for some $G_0 \subseteq G$, see Lemma \ref{lem_divclosed}. If $\Delta (S) \neq \emptyset$, which is the only case that is relevant, then $G_0$ contains a nonzero element $g$ whose order, denote it $d$, divides $n$ and thus is odd. By Lemma \ref{lem_onegenerated_ar} we have $\Delta (\Bc_{\pm} (\{g\})) = \{d - 2\}$. 
Since $\min \Delta (S) \mid \min \Delta (\Bc_{\pm} (\{g\}))  = d - 2$, we have $\min \Delta (S) \in D_2$, establishing the claim. 
\end{proof} 

The description is not always precise. In the following results we pursue the investigations of this.  First, we note that in certain cases the description in fact is complete.

\begin{corollary}
Let $n  \in \mathbb{N}$ such that $n$ and $n-2$ are prime, that is $n$ is the larger one of a couple of  twin primes.  For $H =  \Bc_{\pm} (C_n)$, we have $\Delta^{\ast}   (H) = \{1, n - 2\}$.   
\end{corollary}
\begin{proof}
Since $n$ is prime each non-zero element $C_n$ has order $n$. Since $n-2$ is prime as well, using the notation of Theorem \ref{DeltaAstOdd}, we have $D_2 = \{1, n-2\}$. 
Since by Lemma \ref{min=1} we also have $1 \in \Delta^{\ast}   (H)$, the claim is established.  
\end{proof}

However, if $n-2$ is not prime, we do not know which divisors  of $n-2$ are contained in $\Delta^{\ast} (H)$.  
We show that at least sometimes such divisors occur as a minimal distance. 

\begin{lemma}
Let $n \in \mathbb{N}$ be  odd and of the form $n=a^2 +1$ for $a \in \mathbb{N}$. For $H =  \Bc_{\pm} (C_n)$, we have $a-1 \in \Delta^{\ast}   (H)$.   
\end{lemma}
\begin{proof}
Let $C_n = \langle e \rangle$ and let $G_0 = \{e, ae\}$. We show that for $H' =\Bc_{\pm}(G_0) $  we have $\min \Delta (H')= a-1$. 
By Lemma \ref{minDeltaH} it suffices to study the length of the elements of $\mathcal{A}(H')$.  

We first recall that there are two elements of $\mathcal{A}(H')$ containing only $e$, namely $e^2$ and $e^n$. Likewise there are two elements  containing only $ae$, namely $(ae)^2$ and $(ae)^n$; 
note that the order of $ae$ is $n$, since $a$ and $n$ are co-prime. 

Now assume that $e^v (ae)^w$, with $v, w \in \N$ is in $\mathcal{A}(H')$.

We observe that due to the minimality in a plus-minus weighted zero-sum of this sequence all the weights of $e$ are equal and likewise all the weights of $ae$ are equal. 
This means that either $ve + wae = 0$ or $ve + w(-a)e = 0$. 
In the former case $v = n-wa$ for $w$ from $1$ to $a$ (note that $n - a^2 = 1$) yields indeed minimal plus-minus weighted zero-sum sequences, 
while for $w>a$ the sequence would be divisible by $e (ae)^a$ and hence not minimal.
In the latter case, for $v = a - k$ for $k$ from $1$ to $a-1$ implies that $w = 1 + ka$ and these are indeed  minimal plus-minus weighted zero-sum sequences, 
while other other choices of $v$ and $w$ do not yield minimal plus-minus weighted zero-sum sequences.  

The length of these sequences are $(n - wa) + w  = n - (a-1)w $ for $w$ from $1$ to $a$ and $(a-k) + (1+ka) = a+1+ k(a-1)$ for $k$ from $1$ to $a-1$. 
It follows that  $\gcd \{ |A|-2 \colon A \in \mathcal{A}(H')\} = a-1$; observe that $a-1$ divides   $n-2= a^2 -1$. 
\end{proof}

The problem of determining the set  $\Delta^{\ast}   (H)$ entirely for $H =  \Bc_{\pm} (C_n)$ might be difficult. It certainly is in the case without weights (see \cite{ChapmanetalDeltaCyclic, PlagneSchmidDeltaCyclic}). It seems plausible that results on the minimal distances in the case without weights can be employed to make further progress, yet we do not pursue this line of investigation in this paper.  

For groups of higher rank, we show that  for elementary $p$-groups,  the problem for groups of large rank can be reduced almost entirely to the problem for cyclic groups. 
Again, the exclusion of $0$ in the lemma below is not actual restriction. 

\begin{lemma}
\label{elpDast}
Let $G = C_p^r$ be an elementary $p$-group of rank $r$ for some odd prime $p$. Let $G_0 \subseteq G \setminus \{0\}$ and  $H' =\Bc_{\pm}(G_0)$. If $\min \Delta(H') > 1$, 
then $G_0 =G_1 \uplus \dots \uplus G_k$ where $\langle G_i \rangle$ is cyclic for each $i \in [1, k]$ and $\langle G_0 \rangle = \oplus_{i=1}^k \langle G_i \rangle$.  
\end{lemma}
\begin{proof}
Let  $G_0 \subseteq G$, and let $e_1,\dots, e_k$ be a basis $\langle G_0 \rangle$. 
To show the assertion it suffices to show that if $G_0$ contains some element $e_0 = \sum_{i=1}^k a_ie_i$ 
with $a_i \in [0, p-1]$ and at least two of the $a_i$ non-zero, then $\min \Delta (H') = 1$. 

Without loss we can assume that all the $a_i$ are non-zero and $k \ge 2$. Moreover we can assume that $a_i < p/2$, since we could replace $e_i$ by $(-e_i)$ without affecting the problem. 

Now let $U= e_0e_1^{p-a_1}e_2^{p-a_2}e_3^{a_3} \dots e_k^{a_k}$ and  $V = e_0^2e_1^{p-2a_1}e_2^{p-2a_2}e_3^{2a_3} \dots e_k^{2a_k}$. 
These are both minimal plus-minus weighted zero-sum sequences. We have $U^2 = V e_1^p e_2^p$ and thus the minimal distance is $1$.  
\end{proof}

This lemma shows that $\Delta^{\ast}   (\Bc_{\pm} (C_p^r))$ and $\Delta^{\ast}   (\Bc_{\pm} (C_p))$  are very closely linked.  Before stating this relation in detail, we note  that the argument in the proof of the preceding lemma needs neither the orders of the $e_i$ to be prime nor them to be all equal. We record this as a lemma as it might be useful for future investigations.   
\begin{lemma} 
\label{lem_odd}
Let $G$ be a finite abelian group and let $e_1, \dots, e_k$ be independent elements of odd order,  
say $\ord (e_i) = n_i \ge 3$.    Let $e_0 = a_1 e_1 + \dots  +  a_k e_k$ with $a_i \in [0, n_i - 1]$. Put  $G_0 =  \{e_0, e_1, \dots, e_k\}$ and  $H = \Bc_{\pm} (G_0)$.
If at least two $a_i$ are non-zero, then  $\min \Delta(H) = 1$. 
\end{lemma} 
\begin{proof}
Without loss we can assume that all the $a_i$ are non-zero.  Suppose $k \ge 2$. Moreover we can assume that $a_i < n_i/2$ for each $i$, since we could replace $e_i$ by $(-e_i)$ without affecting the problem. 

Now let $U= e_0e_1^{n_1-a_1}e_2^{n_2-a_2}e_3^{a_3} \dots e_k^{a_k}$ and  $V = e_0^2e_1^{n_1-2a_1}e_2^{n_2-2a_2}e_3^{2a_3} \dots e_k^{2a_k}$. 
These are both minimal plus-minus weighted zero-sum sequences. We have $U^2 = V e_1^{n_1} e_2^{n_2}$ and thus the minimal distance is $1$.   
\end{proof}

We now give the description for elementary $p$-groups. 

\begin{theorem}
\label{thm_elp}
Let $p$ be an odd prime and $r \ge 1$. We have $ \Delta^{\ast}   (\Bc_{\pm} (C_p^r)) = \{ \gcd(d_1, \dots, d_r) \colon d_i \in \Delta^{\ast}   (\Bc_{\pm} (C_p))\}$.  
\end{theorem}
\begin{proof}
By Lemma \ref{elpDast} each subset $G_0$ that yields an element of $\Delta^{\ast}   (\Bc_{\pm} (C_p^r))$ other than $1$ can be written as  $G_0 \setminus\{0\} =G_1 \uplus \dots \uplus G_k$ where $\langle G_i \rangle$ is cyclic for each $i \in [1, k]$ and $\langle G_0 \rangle = \oplus_{i=1}^k \langle G_i \rangle$. 
Now, the minimal distance of $\Bc_{\pm}(G_0)$ is the greatest common divisor of the minimal distances of $\Bc_{\pm}(G_i)$ for $i \in [1,k]$, which shows that each element of $ \Delta^{\ast}   (\Bc_{\pm} (C_p^r))$ is of the claimed form. Conversely, given $d_1, \dots, d_r \in \Delta^{\ast}   (\Bc_{\pm} (C_p))$, there are  $G_1, \dots, G_r \subseteq C_p^r$ such that $\langle G_i \rangle$ is cyclic, $\min\Delta (\Bc_{\pm} (G_i))=d_i $ and $C_p^r=  \langle G_1 \rangle  \oplus \dots \oplus \langle G_r \rangle $. Then $ \min \Bc_{\pm}(G_1 \cup \dots  \cup G_r) = \gcd(d_1, \dots, d_r)$. 
\end{proof}

\subsection{Groups of even order}

In the current subsection we focus on groups of even order and start with a result that allows us to construct sets with relatively large minimal distances. 

\begin{lemma} 
\label{lem_even}
Let $G$ be a finite abelian group and let $e_1, \dots, e_r$ be independent elements of even order, 
say $\ord (e_i) = 2 m_i$. Assume that $m_1 + \dots + m_r \ge 2 $.  Let $e_0 = m_1 e_1 + \dots  +  m_r e_r$,  $G_0 =  \{e_0, e_1, \dots, e_r\}$ and  $H = \Bc_{\pm} (G_0)$.
Then $\Delta (H) = \{ m_1  + \dots  +  m_r - 1\}$. 
\end{lemma} 
\begin{proof} 
We determine $\Ac (H) $. Clearly $e_i^2 \in \Ac (H) $ for each $0 \le i \le r$. Let $A \in \Ac (H)$. If $e_0 \nmid A$, then the distinct elements in $A$ are independent, and $A$ being minimal implies that $A$ is equal to $e_i^k$ for some $1 \le i \le r$. Since the order of $e_i$ is even, it follows that $k=2 $, compare  Lemma \ref{lem_onegenerated_ar}. If $e_0^2\mid A$, then we claim that $e_0^2  = A$. To see this it suffices to note that the order of $e_0$ is $2$, and  thus, $0$ is the only plus-minus weighted sum of  $e_0^2$, which implies that $Ae_0^{- 2}$ has also a plus-minus weighted zero-sum. The only way how this does not contradict the minimality of $A$ is that  $Ae_0^{- 2}$ is empty. 

Thus, it remains to consider the case that  $A$ contains $e_0$ exactly once. We show that in this case $A$ is equal to  $U= e_0 \prod_{i=1}^r e_i^{m_i}$.
We note that $(m_ie_i )  e_i^k$ is a minimal plus-minus weighted zero-sum sequence if and only if $k = m_i$, we refer to Lemma 6.6 in \cite{BMOS} for a detailed argument.    
The claim now follows from the independence of the $e_i$.

Thus we established that  $\Ac (H) = \{U,e_0^2, e_1^2, \dots, e_r^2 \}$.
Now $U^2 = e_0^2 \prod_{i=1}^r (e_i^2)^{m_i}$ is the only relevant relation between these atoms and the claim follows.
\end{proof} 

The argument also shows that $\cb (H) =  m_1  + \dots  +  m_r + 1$ where $\cb (H)$ denotes the catenary degree of $H$. 
It is interesting to note that this value coincides with the value that one expects for many types of groups for the catenary degree of the monoid of classical zero-sum sequences over $\langle e_1, \dots , e_r \rangle$ (see for example \cite{Catenary}). 

We note that if  $e_1, \dots, e_r$ is an independent generating set of elements of even order of $G$, 
say $\ord (e_i) = 2 m_i$, with $\ord (e_i) \mid \ord (e_{i+1})$, then  $\mathsf{D}^{\ast}(\Bc_{\pm} (G)) =  m_1  + \dots  +  m_r +1$.
This is a lower bound for $ \mathsf{D}(\Bc_{\pm} (G))$ that was obtained in \cite{BMOS}.

Thus, for a group with $\mathsf{r}_2(G)=\mathsf{r}(G)$ we have $\mathsf{D}^{\ast}(\Bc_{\pm} (G))-2 \in \Delta^{\ast}(\Bc_{\pm} (G)) $. 
We recall that we always have $\max \Delta^{\ast}(H) \le \mathsf{D}(H)-2$. Thus, if we know that $\mathsf{D}^{\ast}(\Bc_{\pm} (G)) = \mathsf{D}(\Bc_{\pm} (G))$, we get the exact value of the maximum. However, result of this type are rare. Yet, it was established for $C_{2m}$ in \cite{BMOS} and for $C_2 \oplus C_4$, in \cite{Fabsitsetal}. The result on elementary $2$-groups is already in \cite{GeroldingerKainrath}.  
 
\begin{proposition}
\label{prop_DeltaStar}
Let $G$ be a finite abelian group and let $H = \Bc_{\pm} (G)$. 
\begin{enumerate}
\item If $G$ is cyclic of order at least $3$, then $\max \Delta^{\ast}(H) = \mathsf{D}(H)-2$.   
\item If $G$ is an elementary $2$-group of order at least $4$, then $\Delta^{\ast}(H) =[1, \mathsf{D}(H)-2]$. 
\item If $G$ is a group of order at least $4$ exponent a divisor of $4$, then $[1, \mathsf{D}^{\ast}(H)-2] \subseteq \Delta^{\ast}(H) \subseteq [1, \mathsf{D}(H)-2]$. 
\item If $G = C_n \oplus C_{2n}$ for some odd $n \ge 3$, then $  \mathsf{D}^{\ast}(H)-2 \le \max \Delta^{\ast}(H) \le   \mathsf{D}(H)-2$. 
\end{enumerate}
\end{proposition}
\begin{proof}
We recall that we always have $\max \Delta^{\ast}(H) \le \mathsf{D}(H)-2$. 

Let $G= \langle e_1 \rangle $ be cyclic of order $n \ge 3$. If $n $ is odd, then $ \mathsf{D}(H)-2 = n-2$, and $\Delta(\{e\}) = n-2$. Thus the claim follows. 

In all other cases except for 4, we have  $\mathsf{r}_2(G)=\mathsf{r}(G)$ and we have $\mathsf{D}^{\ast}(\Bc_{\pm} (G))-2 \in \Delta^{\ast}(\Bc_{\pm} (G)) $. 
If $n=2m$ is even and $G= C_n$, then  $ \mathsf{D}(H)-2 = m-1$ and $\Delta(\{e, me\}) = m-1$.

If $G = \oplus_{i=1}^r \langle e_i \rangle $ is an elementary $2$-group of rank $r$, then $ \mathsf{D}^{\ast}(H)-2 = r-1$ and thus $r-1$ is an element of $\Delta^{\ast} (H)$.  
Since $H_s =  \Bc_{\pm} (\oplus_{i=1}^s \langle e_i \rangle)$ is a divisor-closed submonoid of $H$ for each $s\le r$ we have that $\Delta(H_s) \subseteq \Delta (H)$. 
Since $\mathsf{D}^{\ast}( \Bc_{\pm} ( \oplus_{i=1}^s \langle e_i \rangle ) ) =s+1$, the claim follows using the same argument as for $s=r$.  

Suppose $G = \oplus_{i=1}^r \langle e_i \rangle $  with $\ord (e_i) \mid \ord (e_{i+1})$ is a group of exponent dividing $4$ and rank $r$. 
As above we have $\mathsf{D}^{\ast}(\Bc_{\pm} (G))-2 \in \Delta^{\ast}(\Bc_{\pm} (G)) $. To complete the argument it suffices to show that 
$G$ has a subgroup $G'$ with $\mathsf{D}^{\ast}(\Bc_{\pm} (G'))=D$ for each $D \in [3, \mathsf{D}^{\ast}(\Bc_{\pm} (G))]$. 
We set  $G_1 = \oplus_{i=1}^{r-1} \langle e_i \rangle \oplus \langle 2e_r \rangle$. 
Then $\mathsf{D}^{\ast}(\Bc_{\pm} (G_1)) = \mathsf{D}^{\ast}(\Bc_{\pm} (G))-1$. Now the argument follows by a direct induction. 

We show the fourth part. Let $C_n \oplus C_{2n}= \langle e_1 \rangle \oplus \langle e_2 \rangle$ with $\ord(e_1)= n$ and $\ord(e_2)=2n$. 
We consider $G_0 =\{e_1+e_2, e_2 \}$. We determine $\Ac( \Bc_{\pm} (G_0))$. Clearly we have $V_1= (e_1+e_2)^2$ and $V_2= e_2^2$ and no other elements with only one of the two elements. In addition we find $U=(e_1 + e_2)^n e_2^n$. We have $U^2 = V_1^n V_2^n$ and indeed $\mathsf{L}(U^2)= \{2, 2n\}$, and we find that $\Delta ( \Bc_{\pm} (G_0) )  = \{2n-2\}$. We recall that  $\mathsf{D}^{\ast}(H) = n-1 +(2n)/2 + 1=2n$, establishing the claim.  
\end{proof}

In the following result we characterize the groups for which $\max \Delta^{\ast}(H)$ is small. 
\begin{proposition}
Let $G$ be a finite abelian group and let $H = \Bc_{\pm} (G)$. 
\begin{enumerate}
\item Then $\max \Delta^{\ast}(H) = 1$ if and only if $\exp(G) = 3$ or $G=C_2^2$ or $G=C_4$. 
\item Then $\max \Delta^{\ast}(H) = 2$ if and only  $G=C_2^3$ or $G=C_2 \oplus C_4$.
\end{enumerate}
\end{proposition}
\begin{proof}
If $|G| \le 2$, then $\Delta^{\ast}(H) = \emptyset$ by Lemma \ref{min=1} and thus $\max \Delta^{\ast}(H) \neq 1$. 
Thus assume $|G|\ge 3$ and let $n = \exp(G)$. 

Supose $n$ is odd.  Then, by Proposition \ref{prop_DeltaStar} $\max \Delta^{\ast}(H) = n-2$. Therefore,  $\max \Delta^{\ast}(H) = 1$ if and only if and $n=3$, 
and $\max \Delta^{\ast}(H) = 2$ is impossible. 

Now, suppose that $n$ is even. Then, by  Proposition \ref{prop_DeltaStar},  $\Delta^{\ast}(H) = [1, r-1]$, where $r$ denotes the rank of $G$. 
Thus $\max \Delta^{\ast}(H) = 1$ if and only if $G=C_2^2$, and  $\max \Delta^{\ast}(H) =2$ if and only if $G=C_2^3$. 

If $n \ge 4$, by Lemma \ref{lem_even}, we have $n/2 -1 \in  \Delta^{\ast}(H)$. 
Thus, $\max \Delta^{\ast}(H) = 1$ implies that $n = 4$, and  $\max \Delta^{\ast}(H) = 2$ implies $n \le 6$. 
If $n=4$, then Lemma \ref{lem_even} shows that $\max \Delta^{\ast}(H) \ge \mathsf{r}_2(G)  + \mathsf{r}_4(G) -1 $.  
Therefore, $\max \Delta^{\ast}(H) = 1$, implies $G=C_4$, by Proposition \ref{prop_DeltaStar} indeed $\Delta^{\ast}(\Bc_{\pm} (C_4)) = \{1\}$, the claim is established.
Moreover $\max \Delta^{\ast}(H) = 2$, implies $G=C_2 \oplus C_4$. Since $\mathsf{D}(H) = 4$ in that case,  indeed $\max \Delta^{\ast}(\Bc_{\pm} (C_2 \oplus C_4)) = 2$. 
\end{proof}

We end the section with a result that might seem peculiar at first, but is highly useful in the context of the subsequent section. 
\begin{theorem}
\label{ParityDelta} 
Let $G$ be a finite abelian group with $|G|\ge 5$. The following statements are equivalent. 
\begin{enumerate}
\item $|G|$ is even. 
\item $\Delta^{\ast}(H)$ contains an even element. 
\item $\Delta_1(H)$ contains an even element.
\end{enumerate}
\end{theorem}

\begin{proof}
The equivalence of the second and third statement is clear by the inclusions between $\Delta^{\ast}(H)$ and $\Delta_1(H)$ and the divors of  $\Delta^{\ast}(H)$  recalled in Section \ref{sec_SSL}. 
We show that the first and second statement are equivalent. 

Assume that $\Delta^{\ast}(H)$ contains an even element. Then $|G|$ cannot be odd, as this would contradict Theorem \ref{DeltaAstOdd} as all elements of $D_2$ are odd.
It remains to show that if $|G|$ is even, then $\Delta^{\ast}(H)$ contains an even element.
If $G$ is not a $2$-group, then $G$ contains an element of order $2m$ for some odd $m \ge 2$. 
By Lemma \ref{lem_even} we know that $ m - 1  \in \Delta^{\ast}(H)$, and this is even. 

Thus assume that $G$ is a $2$-group. If the rank of $G$ is at least $3$, then by Proposition \ref{prop_DeltaStar} $\Delta^{\ast}(H)$ contains $2$, 
and if the rank is $2$ then by Proposition \ref{prop_DeltaStar} it contains $1+ (\exp(G)/2 - 1) = \exp(G)/2$, which is even as $\exp(G)\neq 2$ by the assumption $|G| \ge 5$. 

It remains to study the case of cyclic $2$-groups. We note that for $e$ an element of order $8$, the  monoid $\Bc_{\pm}(\{e, 3e\})$ 
has minimal distance $2$; note that the irreducibles are $e^2, (3e)^2, e^3(3e), e(3e)^3$. 
\end{proof}

\subsection{Groups of infinite order}

In this section, we gather a few results on groups of infinite order, which basically reduce the problem of determining the set of minimal distances to the one for finite groups. 

\begin{theorem} 
\label{DeltaStar_inf}
Let $G$ be a an abelian group.  
\begin{enumerate}
\item If $G$ is not a torsion group, then $\Delta^{\ast}(\Bc_{\pm}(G)) = \mathbb{N}$. 
\item If $G$ is a torsion group, then $\Delta^{\ast}(\Bc_{\pm}(G)) = \bigcup_{G' \subseteq G \text{ finite subgroup}} \Delta^{\ast}(\Bc_{\pm}(G'))$.
\end{enumerate}
\end{theorem}
\begin{proof}
To prove the first claim, let $g \in G$ be an element of infinite order. We assert that for each $n \in \mathbb{N}$, one has  $\Delta(\Bc_{\pm}( \{g, (n+1)g\})) = n$.
One checks that $ \Ac(\Bc_{\pm}( \{g, (n+1)g\})) = \{g^2, ((n+1)g)^2, g^{n+1}((n+1)g)\}$. 
Noting that   $ (g^{n+1}((n+1)g))^2 = g^2 ((n+1)g)^2$ is the only relevant relation, the claim follows. 
 
We now turn to the second claim. It is clear that 
\[\Delta^{\ast}(\Bc_{\pm}(G)) \supset \bigcup_{G' \subseteq G \text{ finite subgroup}} \Delta^{\ast}(\Bc_{\pm}(G')).\] 
Let $d \in \Delta^{\ast}(\Bc_{\pm}(G))$. We need to show that there is a finite subgroup $G'$ of $G$ such that $d \in  \Delta^{\ast}(\Bc_{\pm}(G'))$. 
Let $G_0  \subseteq G$ such that $d = \min \Delta (\Bc_{\pm}(G_0))$. Let $B \in \Bc_{\pm}(G_0)$ such that $d\in \Delta (B)$. Then $d = \min \Delta (\Bc_{\pm}(\supp(B)))$ holds as well.  
Let $G'  = \langle \supp (B) \rangle$, which is a finite subgroup of $G$ as $\supp (B)$ is finite and the group is a torsion group. 
Since $\supp(B) \subseteq G'$, it follows that $d \in  \Delta^{\ast}(\Bc_{\pm}(G'))$. 
\end{proof}

\begin{corollary}
Let $G$ be an elementary $p$-group of infinite rank. Let $H = \Bc_{\pm}(G)$.  
\begin{enumerate}
\item If $p=2$, then $\Delta^{\ast}(H) = \mathbb{N}$. 
\item If $p$ is odd, then $ \Delta^{\ast}   (H) =  \{\gcd(d_1, \dots, d_s) \colon s\in \N, \, d_i \in  \Delta^{\ast}   (\Bc_{\pm}(C_p)) \}  $.   
\end{enumerate}
\end{corollary}
\begin{proof}
The result is a direct consequence of  Theorem \ref{DeltaStar_inf} and the results on finite elementary $p$-groups obtained in Theorem \ref{thm_elp} and Proposition \ref{prop_DeltaStar}. 
\end{proof}

\begin{corollary}
Let $G$ be an abelian torsion group without an element of even order. Let $D_1$ denote the set of all $d \in \N$ such that $G$ contains an element of order $d+2$ and let $D_2 \subseteq \N$ denote the set of all $d'$ that divide an element of $D_1$.  
Then $D_1   \subseteq \Delta^{\ast}(\Bc(G)) \subseteq D_2$. 
\end{corollary} 
\begin{proof}
The result is a direct consequence of  Theorem \ref{DeltaStar_inf} and  Theorem \ref{DeltaAstOdd}.
\end{proof}

\section{Applications to the characterization problem}

It is obvious that isomorphic abelian groups $G$ and $G'$ yield isomorphic monoids of zero-sum sequences $\mathcal{B}(G)$ and $\mathcal{B}(G')$, and hence also the systems of sets of lengths are equal  $\mathcal{L}(\mathcal{B}(G)) = \mathcal{L}(\mathcal{B}(G'))$. The same hold true if we consider plus-minus weighted zero-sum sequences (or in fact weighted zero-sum sequences for any choice of weights). 

A natural question to ask is to what extent these implications can be reversed. For the first implication this is called the Isomophism Problem,  that is it is the problem of determinig if isomorphy of the monoids of (weighted) zero-sum sequences implies isomorphy of the underlying groups, the problem is completely answered for abelian groups in the case without weights and for groups that are the direct sum of cyclic groups in the case with  plus-minus weights (see \cite{Fabsitsetal} and the references there). 

For the question involving length this is called the Characterization Problem. In its classical form it can be stated as asking whether  
for $G$ and $G'$  finite abelian groups such that $\mathcal{L}(\mathcal{B}(G)) = \mathcal{L}(\mathcal{B}(G'))$,  
it follows that $G$ and $G'$ are isomorphic.  The history of this problem goes back more than fifty years and initially arose in the context of obtaining arithmetic characterizations of the class group of a ring of algebraic integers, whence the name (see \cite{Narkiewicz}). 
We mention that the condition that the group is finite is crucial. All infinite abelian groups yield the same system of sets of lengths, which differs from that of every finite group. This is known as Kainrath's Theorem and we discussed it in Section \ref{sec_SSL}. 

It is well-known and easy to see that $C_1$ and $C_2$ both give rise to half-factorial monoids of zero-sum sequences, indeed the monoids are even free. This simple example apart, there is one more pair of non-isomorphic groups known that yield the same system of set of lengths for their monoids of zero-sum sequences, namely $C_3$ and $C_2^2$.   
The standing conjecture is that these are the only cases where this arises and that for each finite abelian group $G$ that is not isomorphic to one of those groups, a condition that ensures this is to assume that $\mathsf{D}(\mathcal{B}(G)) \ge 4$, it is indeed true that  $\mathcal{L}(\mathcal{B}(G)) = \mathcal{L}(\mathcal{B}(G'))$ implies that $G$ and $G'$ are isomorphic. 
We refer to \cite{CharSurvey}  for an overview of older results on the subject and to \cite{GerSchChar,ZhongChar} for more recent contributions. 

A first investigation of the  Characterization Problem for plus-minus weighted zero-sum problems, and more broadly of the problem of describing the system of sets of lengths for plus-minus weighted zero-sum problems has been undertaken in \cite{Fabsitsetal}. 
Very recently, as already mentioned in Section \ref{sec_SSL}, a result like Kainrath's Theorem has also been obtained for the monoids of plus-minus weighted zero-sum sequences \cite{GeroldingerKainrath}. 

As in the case without weights $C_1$ and $C_2$ both give rise to half-factorial monoids of zero-sum sequences, indeed even free monoids.
For other groups of small order in \cite{Fabsitsetal} the authors proved the following result. 

\begin{theorem}
$\mathcal{L}(\Bc_{\pm} (C_2^2))=\mathcal{L}(\Bc_{\pm} (C_3)) = \mathcal{L}(\Bc_{\pm} (C_4))=\{y+2k + [0,k] \colon y,k \in \mathbb{N}_0\}$
Conversely, if $\mathcal{L}(\Bc_{\pm} (G))= \{y+2k + [0,k] \colon y,k \in \mathbb{N}_0\}$ for some finite abelian group $G$, then $G$ is isomorphic to one of the three groups $C_2^2, C_3, C_4$. 
\end{theorem}
Note that the sets $y +2k + [0,k]$ are just the sets $\{m, m+1, m+2, \dots, n\}$ with $n/m \le 3/2$.  Moreover the authors settled the Characterization Problem for cyclic groups of odd order.

In the current section, we try to make further progress on this problem using our results on $\Delta^{\ast}$.  
It is in most cases quite difficult to describe $\mathcal{L}(H)$ completely. In order to solve, the characterization problem one often proceeds like this. 
One establishes that if $\mathcal{L}(H) = \mathcal{L}(H')$, then certain arithmetic invariants of $H$ and $H'$ are equal. Using results on these arithmetic invariants, one obtains information on $H$ and $H'$. 

For invariants that are essentially defined directly via the system of sets of lengths, such as $\rho$, $\rho_k$, $\Delta$, and $\Delta_1$ it is clear that equality of  $\mathcal{L}(H)$  and $\mathcal{L}(H')$ implies the equality of the respective invariants for $H$ and $H'$. 
Since for $H= \mathcal{B}_{\pm}(G)$ and $H'= \mathcal{B}_{\pm}(G')$, it is known that   $\rho_2(H) = \mathsf{D}(H)$ and $\rho_2(H') = \mathsf{D}(H')$ (see \cite{BMOS}), it follows that the equality of $\mathcal{L}(H)$  and $\mathcal{L}(H')$ also implies that $\mathsf{D}(H)= \mathsf{D}(H')$. Since we use the result, we also recall that there it was proved that in case $|G|$ is odd, one has  $\mathsf{D}(H) = \mathsf{D}(G)$.

Likewise the fact that 
\[\Delta^{\ast}(H) \subseteq  \Delta_1(H) \subseteq \{d' \mid d \colon d\in \Delta^{\ast}(H) \},\] implies that in case 
$\mathcal{L}(H) = \mathcal{L}(H')$ also the sets $\Delta^{\ast}(H)$ and $\Delta^{\ast}(H')$ are closely linked: 
$\max \Delta^{\ast}(H) = \max \Delta^{\ast}(H')$ and more precisely, if $M$ denotes this maximum, the sets  $\Delta^{\ast}(H)\cap \mathbb{N}_{>M/2}$ and $\Delta^{\ast}(H')\cap \mathbb{N}_{>M/2}$ are equal.   

Then, one uses results that give the values of these arithmetical invariants in terms of algebraic invariants of the groups, to establish conditions on the groups, and ideally their isomorphy.

\begin{theorem}
\label{thm_char_gen}
Let $G$ be a finite abelian group of odd order and let $G'$ be a finite abelian group. We set $H = \mathcal{B}_{\pm}(G)$ and $H' = \mathcal{B}_{\pm}(G')$. 
If $\mathcal{L}(H) = \mathcal{L}(H')$, then $\exp(G) = \exp(G')$ and $\mathsf{D}(G) = \mathsf{D}(G')$. 
\end{theorem}
\begin{proof}
We know that $\rho(H)= \rho(H')$ and $\Delta_1(H)= \Delta_1(H')$. 
By Theorem \ref{DeltaAstOdd} we know that $\max \Delta_1(H) = \exp(G)-2$.

By Theorem \ref{ParityDelta} we know that $\Delta_1(H)$ only contains odd elements. 
Thus, $\Delta_1(H')$ only contains odd elements as well, and $G'$ is a group of odd order too. Thus, $\max \Delta_1(H') = \exp(G') -2$ as well and $\exp(G) = \exp(G')$.

To obtain the result on the Davenport constants it suffices to recall that $\rho_2(H)= \mathsf{D}(H )$ and $\rho_2(H')= \mathsf{D}(H ')$, and since $|G|$ and $|G'|$ are odd 
we have $\mathsf{D}(H ) = \mathsf{D}(G)$ and $\mathsf{D}(H ') = \mathsf{D}(G')$.

\end{proof}

We insist that indeed the classical Davenport constants of $G$ and $G'$ are equal as for groups of odd order or equivalently odd exponent, as recalled above, they are equal to  $\mathsf{D}(\mathcal{B}_{\pm}(G))$ and  $\mathsf{D}(\mathcal{B}_{\pm}(G'))$ respectively. Below we freely use the results on the Davenport constant recalled in Section \ref{sec_SSL}. 

This result allows us to prove that elementary $p$-groups  for $p$ an  odd prime are characterized by the system of sets of length of for plus-minus weighted zero-sum sequences. Since the same reasoning allows us to characterize certain other $p$-groups and groups of odd exponent, we start with a technical proposition. 

\begin{proposition}
\label{prop_char_dav}
Let $G$ be a finite abelian group of exponent $n$ odd  and Davenport constant $D$ such that there is no group non-isomorphic to $G$ that also has exponent $n$ and Davenport constant $D$. 
If  $\mathcal{L}(\mathcal{B}_{\pm}(G)) = \mathcal{L}(\mathcal{B}_{\pm}(G'))$, then $G \cong G'$.
\end{proposition}
\begin{proof}
By Theorem \ref{thm_char_gen} we know that the exponents of $G$ and $G'$ are equal and that the Davenport constants of $G$ and $G'$ are equal. 
The claim is thus obviously true by the assumption on $G$. 
\end{proof}

We first note that this result yields an alternative proof for the characterization of cyclic groups of odd order already obtained in \cite{Fabsitsetal}. 

\begin{corollary}
Let $G$ be a cyclic group of odd order.   
If  $\mathcal{L}(\mathcal{B}_{\pm}(G)) = \mathcal{L}(\mathcal{B}_{\pm}(G'))$, then $G \cong G'$.
\end{corollary}
\begin{proof}
A cyclic group fulfills the condition of Proposition \ref{prop_char_dav} as cyclic groups are the only groups for which the exponent equals the Davenport constant.
\end{proof}

We now establish the result for elementary $p$-groups of odd order.

\begin{corollary}
Let $G$ be an elementary $p$-group for an odd prime $p$.   
If  $\mathcal{L}(\mathcal{B}_{\pm}(G)) = \mathcal{L}(\mathcal{B}_{\pm}(G'))$, then $G \cong G'$.
\end{corollary}
\begin{proof}
For a group $G$ with $\exp(G)$ prime, the Davenport constant is given by $1 + \mathsf{r}(G) (\exp(G)-1)$ and thus elementary $p$-groups fulfill the condition of Proposition \ref{prop_char_dav}.
\end{proof}
 
We briefly discuss the problem for elementary $2$-groups at the end of the section. Yet, before this we establish two results that can be obtained via our method. 

\begin{corollary}
Let $G= C_p \oplus C_{n}$ where  $n$ is odd and $p$ is the smallest prime divisor of $n$.       
If  $\mathcal{L}(\mathcal{B}_{\pm}(G)) = \mathcal{L}(\mathcal{B}_{\pm}(G'))$, then $G \cong G'$.
\end{corollary}
\begin{proof}
It suffice to show that the groups fulfill the condition of Proposition \ref{prop_char_dav}. First recall that the Davenport constant of $G$ is $n + p-1$, as the group has rank two. 
Now, if $G'$ is a group with exponent $n$, then either $G'$ is cyclic or of the form $H\oplus C_n$ where $1 \neq \exp(H)\mid n$. 
In the former case the Davenport of $G'$ is $n$ and thus not equal to that of $G$ in the latter case the Davenport constant is at least $n +  \mathsf{D}(H)-1 \ge n +  \exp(H)-1$ and thus strictly greater than $n+p-1$ unless $\exp(H)=p$ and $\mathsf{D}(H) = \exp(H)$. Thus we get that $H$ is cyclic of order $p$, implying the claim. 
\end{proof}

To obtain a result for $p$-groups, we establish a simple lemma on the values of the Davenport constant of $p$-groups.

\begin{lemma}
\label{lem_Dp^2}
Let $K$ be a $p$-group with Davenport constant $D$. 
\begin{enumerate}
\item If $D < p^2$ and $K'$ is another $p$-group with Davenport constant $D$, then $K$ and $K'$ are isomorphic. 
\item If $D \ge p^2$ then there exists a $p$-group $K'$ with Davenport constant $D$ such that $K$ and $K'$ are not isomorphic. 
\end{enumerate}
\end{lemma}
\begin{proof}
To see the first point, it suffices to note that $D < p^2$ implies $\exp(K) < p^2$ and $\exp(K') < p^2$. Thus, both have exponent $p$  and the claim follows directly. 

For the second point, we distinguish two cases. Assume that the exponent of $K$ is $p$, say $K \simeq C_p^r$. By the assumption on the Davenport constant we get that $r \ge p+1$. 
Yet then $K' = C_p^{r-p-1} \oplus C_{p^2}$ has the same Davenport constant.    

Now, assume that the exponent of $K$ is not prime, say, it is $p^k$ with $k > 1$. Let $K_0$  be a group such that $K \simeq K_0 \oplus C_{p^k}$.
Then set $K' = K_0 \oplus C_{p}^{(p^k -1)/(p-1)}$, which has the same Davenport constant.  
\end{proof}

\begin{corollary}
Let $G$ be a $p$-group for $p$ odd and assume that $\mathsf{D}(G) < \exp(G)  + p^2 -1$.   
If  $\mathcal{L}(\mathcal{B}_{\pm}(G)) = \mathcal{L}(\mathcal{B}_{\pm}(G'))$, then $G \cong G'$.
\end{corollary}
\begin{proof}
To proof the corollary, we show that a group $G$ that fulfills the conditions in this corollary, indeed fulfills the conditions in Proposition \ref{prop_char_dav}, which directly yields the result. 
That is, we need to show that for such a $G$ there are no other groups (up to isomorphism) with the same exponent and Davenport constant. Indeed, let $K$ be a group such that $G \simeq K \oplus C_{\exp(G)}$ fulfills the conditions of the corollary. Then $\mathsf{D}(G) = \exp(G) + \mathsf{D}(K) - 1$, as $G$ is a $p$-group. 
It follows that $\mathsf{D}(K) < p^2$,  and thus by Lemma \ref{lem_Dp^2} the group $K$, which is a $p$-group, is determined (up to isomorphism) by its Davenport constant.  
\end{proof}

We did not make progress on the characterization problem for groups of even order as our results on $\Delta^{\ast}(\Bc_{\pm}(G))$ are weaker in that case. 
Moreover, even in a case where we have good knowledge of $\Delta^{\ast}(\Bc_{\pm}(G))$, such as elementary $2$-groups it is not clear how to exploit it, 
and characterizing elementary $2$-groups might actually be difficult. It is worth noting that in \cite{Fabsitsetal} the characterization problem was solved for $C_2^3$ and $C_2 \oplus C_4$, yet needs a detailed analysis of the respective systems of sets of lengths.

\section*{Acknowledgments} The research of W. A. Schmid is supported by the  project ANR-21-CE39-0009 - BARRACUDA. The authors thank the referee for numerous helpful remarks.

\end{document}